\documentclass[a4paper]{amsart}

\pdfoutput=1

\usepackage[utf8]{inputenc}
\usepackage[T1]{fontenc}
\usepackage{lmodern}
\usepackage{amsthm, amssymb, amsmath, amsfonts, mathrsfs}
\usepackage[colorlinks=true, pdfstartview=FitV, linkcolor=blue, citecolor=blue, urlcolor=blue,pagebackref=false]{hyperref}

\usepackage{MnSymbol} 

\usepackage{esint} 

\usepackage{microtype}

\newtheorem{proposition}{Proposition}
\newtheorem{theorem}[proposition]{Theorem}
\newtheorem{lemma}[proposition]{Lemma}

\theoremstyle{remark}
\newtheorem{remark}[proposition]{Remark}

\theoremstyle{definition}

\numberwithin{equation}{section}
\numberwithin{proposition}{section}

\renewcommand{\leq}{\leqslant}
\renewcommand{\geq}{\geqslant}
\newcommand{\ls}{\lesssim}
\newcommand{\les}{\ls}

\newcommand{\B}{\mathbf{B}}

\newcommand{\E}{\mathbb{E}}

\newcommand{\R}{\mathbb{R}}
\newcommand{\C}{\mathbb{C}}

\renewcommand{\tilde}{\widetilde}
\newcommand{\eps}{\varepsilon}

\newcommand{\Pb}{\mathbb{P}}
\newcommand{\what}{\widehat}

\newcommand{\F}{\mathcal{F}}

\title[1D Schr\"odinger equation with spacetime white noise]{The 1D Schr\"odinger equation with a spacetime white noise: the average wave function}
\author{Yu Gu}

\address[Yu Gu]{Department of Mathematical Sciences, Carnegie Mellon University, Pittsburgh, PA 15213, USA, email: yug2@andrew.cmu.edu}

\begin{document}

\begin{abstract}
For the 1D Schr\"odinger equation with a mollified spacetime white noise, we show that the average wave function converges to the Schr\"odinger equation with an effective potential after an appropriate renormalization.

\bigskip

%

\noindent \textsc{Keywords:} random Schr\"odinger equation, renormalization, path integral.

\end{abstract}
\maketitle
%
%
%
%
%
%
%
%
\section{Main result}
\label{s.intro}

Consider the Schr\"odinger equation driven by a weak stationary spacetime Gaussian potential $V(t,x)$:
\begin{equation}\label{e.maineq1}
i\partial_t\phi(t,x)+\frac12\Delta\phi(t,x)-\sqrt{\eps}V(t,x)\phi(t,x)=0,   \    \  t>0, x\in\R,
\end{equation}
on the diffusive scale $(t,x)\mapsto (\tfrac{t}{\eps^2},\tfrac{x}{\eps})$,
\begin{equation}\label{e.rescale}
\phi_\eps(t,x):=\phi(\frac{t}{\eps^2},\frac{x}{\eps})
\end{equation}
satisfies
\begin{equation}\label{e.maineq2}
i\partial_t\phi_\eps(t,x)+\frac12\Delta\phi_\eps(t,x)-\frac{1}{\eps^{3/2}}V(\frac{t}{\eps^2},\frac{x}{\eps})\phi_\eps(t,x)=0.
\end{equation}
With appropriate decorrelating assumptions on $V$, the rescaled large highly oscillatory potential $\eps^{-3/2}V(t/\eps^2,x/\eps)$ converges in distribution to a spacetime white noise, denoted by $\dot{W}(t,x)$. To the best of our knowledge, the asymptotics of $\phi_\eps$ and making sense of the limit of \eqref{e.maineq2}, which formally reads 
\[
i\partial_t \Phi(t,x)+\frac12\Delta\Phi(t,x)-\dot{W}(t,x)\Phi(t,x)=0, 
\]
is an open problem. The goal of this short note is to take a first step by analyzing $\E[\phi_\eps]$ as $\eps\to0$.

\subsection{Assumptions on the randomness}
We assume the spacetime white noise $\dot{W}(t,x)$ is built on the probability space $(\Omega,\mathcal{F},\Pb)$, and 
\[
V(t,x)=\int_{\R^2} \varrho(t-s,x-y)\dot{W}(s,y)dyds
\]
for some mollifier $\varrho\in \mathcal{C}_c^\infty$ with $\int \varrho=1$. By the scaling property of $\dot{W}$, we have  
\[
\begin{aligned}
\frac{1}{\eps^{3/2}}V(\frac{t}{\eps^2},\frac{x}{\eps})=&\frac{1}{\eps^{3/2}}\int_{\R^2}\varrho(\frac{t}{\eps^2}-s,\frac{x}{\eps}-y)\dot{W}(s,y)dyds\\
=&\frac{1}{\eps^{3/2}}\int_{\R^2}\frac{1}{\eps^3}\varrho(\frac{t-s}{\eps^2},\frac{x-y}{\eps})\dot{W}(\frac{s}{\eps^2},\frac{y}{\eps})dyds\\
 \stackrel{\text{law}}{=} &\int_{\R^2}\frac{1}{\eps^3}\varrho(\frac{t-s}{\eps^2},\frac{x-y}{\eps})\dot{W}(s,y)dyds,
\end{aligned}
\]
which converges in distribution to $\dot{W}$ independent of the choice of $\varrho$. For simplicity, we choose 
\[
\varrho(t,x)= \frac{\eta(t)}{\sqrt{\pi}}e^{-x^2},
\]
with $\eta\in\mathcal{C}_c^\infty(\R)$ and $\int \eta=1$. The covariance function of $V$ is 
\begin{equation}\label{e.cov}
R(t,x)=\E[V(t,x)V(0,0)]=\int_{\R^2} \varrho(t+s,x+y)\varrho(s,y)dyds=R_\eta(t)q(x),
\end{equation}
with 
\begin{equation}\label{e.defq}
R_\eta(t):=\int_\R\eta(t+s)\eta(s)ds,  \    \ q(x):=\frac{1}{\sqrt{2\pi}}e^{-\frac{x^2}{2}}.
\end{equation}
We define $\tilde{R}(\omega,\xi)$ as the Fourier transform of $R$ in $(t,x)$:
\[
\tilde{R}(\omega,\xi)=\int_{\R^2} R(t,x)e^{-i\omega t-i\xi x } dtdx.
\]
We use $\what{f}$ to denote the Fourier transform of $f$ in the $x$ variable:
\[
\what{f}(\xi)=\int_{\R} f(x)e^{-i\xi x}dx.
\]

\subsection{Main result}

Assuming the initial data $\phi_\eps(0,x)=\phi_0(x)\in\mathcal{C}_c^\infty(\R)$, so we have a low frequency wave before rescaling: $\phi(0,x)=\phi_0(\eps x)$. The following is the main result:
\begin{theorem}\label{t.mainth}
There exists $z_1,z_2\in \C$ depending on the mollifier $\varrho$, given by \eqref{e.defz1} and \eqref{e.defz2} respectively, such that for any $t>0,\xi\in\R$, 
\begin{equation}\label{e.mainresult}
\E[\what{\phi}_\eps(t,\xi)]e^{\frac{z_1t}{\eps}}\to \what{\phi}_0(\xi) e^{-\frac{i}{2}|\xi|^2 t+z_2t},  \quad \mbox{ as } \eps\to0.
\end{equation}
\end{theorem}
\bigskip
We make a few remarks.

\begin{remark}
The limit in \eqref{e.mainresult} is the solution to 
\[
i\partial_t \bar{\phi}+\frac12\Delta\bar{\phi}-iz_2\bar{\phi}=0,  \   \ \bar{\phi}(0,x)=\phi_0(x),
\]
written in the Fourier domain:
\[
\bar{\phi}(t,x)=\frac{1}{2\pi}\int_\R \what{\phi}_0(\xi) e^{-\frac{i}{2}|\xi|^2 t+z_2t}e^{i\xi x}d\xi.
\]
\end{remark}

\begin{remark}
In the parabolic setting, a Wong-Zakai theorem is proved \cite{chandra2017moment,gu2018another,hairer2015multiplicative,hairer2015wong} for 
\[
\partial_t u_\eps=\frac12\Delta u_\eps+\frac{1}{\eps^{3/2}}V(\frac{t}{\eps^2},\frac{x}{\eps})u_\eps, \  \ u(0,x)=u_0(x).
\]
The result says that there exists $c_1,c_2>0$ depending on $\varrho$ such that 
\begin{equation}\label{e.conshe}
u_\eps(t,x) e^{-\frac{c_1t}{\eps}-c_2t}\Rightarrow \mathcal{U}(t,x) \quad \mbox{ in distribution},
\end{equation}
 where $\mathcal{U}$ solves the stochastic heat equation with a multiplicative spacetime white noise
\[
\partial_t \mathcal{U}(t,x)=\frac12\Delta \mathcal{U}(t,x)+\mathcal{U}(t,x)\dot{W}(t,x),  \    \  \mathcal{U}(0,x)=u_0(x),
\]
with the product $\mathcal{U}(t,x)\dot{W}(t,x)$ interpreted in the It\^o's sense. Writing the above equation in the mild formulation, it is easy to see that $\E[\mathcal{U}]$ solves the unperturbed heat equation
\[
\E[\mathcal{U}(t,x)]=\int_\R \frac{1}{\sqrt{2\pi t}}e^{-\frac{|x-y|^2}{2t}}u_0(y)dy,
\]
thus, a consequence of \eqref{e.conshe} is 
\[
\E[\what{u}_\eps(t,\xi)]e^{-\frac{c_1t}{\eps}-c_2t}\to \what{u}_0(\xi)e^{-\frac12|\xi|^2t},
\]
which should be compared to \eqref{e.mainresult} in the Schr\"odinger setting, with $-c_1,c_2$ corresponding to $z_1,z_2$.
\end{remark}

\begin{remark}\label{r.kinetic}
Starting from the microscopic dynamics \eqref{e.maineq1}, if we consider a time scale that is shorter than the one in \eqref{e.rescale}, and a low frequency initial data 
\[
(t,x)\mapsto (\frac{t}{\eps},\frac{x}{\sqrt{\eps}}), \quad \phi(0,x)=\phi_0(\sqrt{\eps} x),
\]  a homogenization result was proved in \cite{gu2016random}: for any $t>0,\xi\in\R$,
\begin{equation}\label{e.homo}
\eps^{\frac{d}{2}}\what{\phi}(\frac{t}{\eps},\sqrt{\eps}\xi)\to \what{\phi}_0(\xi)e^{-\frac{i}{2}|\xi|^2t-z_1t}
\end{equation}
in probability. Here $z_1$ is the same constant as in Theorem~\ref{t.mainth}. If we instead consider a high frequency initial data $\phi(0,x)=\phi_0(x)$, which varies on the same scale as the random media, a kinetic equation was derived on the time scale of $\tfrac{t}{\eps}$ in \cite{bal2011asymptotics}:
\begin{equation}\label{e.kinetic}
\E[|\what{\phi}(\frac{t}{\eps},\xi)|^2]\to \overline{W}(t,\xi),
\end{equation}
where $\overline{W}(t,\xi)=\int_\R W(t,x,\xi)dx$ and $W$ solves the radiative transfer equation
\begin{equation}\label{e.wigner}
\partial_t W(t,x,\xi)+\xi\cdot \nabla_x W(t,x,\xi)=\int_{\R} \tilde{R}(\frac{|p|^2-|\xi|^2}{2},p-\xi) (W(t,x,p)-W(t,x,\xi))\frac{dp}{2\pi}.
\end{equation}
For similar results in the case of a spatial randomness, see \cite{chen2018weak,erdHos2000linear,spohn1977derivation}. The equation \eqref{e.wigner} shows that, in the high frequency regime where the wave and the random media interact fully, the momentum variable follows a jump process with the kernel given by $\tilde{R}(\tfrac{|p|^2-|\xi|^2}{2},p-\xi)$. The real part of the constant $z_1$, in \eqref{e.mainresult} and \eqref{e.homo}, describes the total scattering cross-section, i.e., the jumping rate at the zero frequency:
\[
2\mathrm{Re}[z_1]=\int_{\R} \tilde{R}(\frac{|p|^2}{2},p)\frac{dp}{2\pi}.
\]
Thus, the renormalization in \eqref{e.mainresult} can be viewed as a compensation of the exponential attenuation of wave propagation on the time scale of $\tfrac{t}{\eps^2}$. We emphasize that the average wave function (more precisely, the term $\E[\what{\phi}_\eps(t,\xi)]\E[\what{\phi}_\eps^*(t,\xi)]$) only captures the ballistic component of wave.
 \end{remark}
 
 \begin{remark}
It is unclear at this stage what explicit information the convergence in  \eqref{e.mainresult} implies. On one hand, if we expect the family of random variables $\{\what{\phi}_\eps(t,\xi)\}_{\eps\in(0,1)}$ to converge in distribution to some random limit after any possible renormalization and assume the \emph{uniform integrability}, then our result shows that $e^{z_1t/\eps}$ is the only possible renormalization factor since the uniform integrability ensures the mean $\E[\what{\phi}_\eps(t,\xi)]$ also converges after the same rescaling. On the other hand, without the uniform integrability it is a priori unclear whether the convergence of $\what{\phi}_\eps(t,\xi)$ is related to the convergence of its first moment. In addition, based on the discussion in Remark~\ref{r.kinetic}, we know that the wave field $\what{\phi}(t,\xi)$ decays exponentially on the time scale $t/\eps^2$ because $\mathrm{Re}[z_1]>0$, and the lost energy escapes to high frequency regime through multiple scatterings. From this perspective, the physical meaning is unclear when we multiply the exponentially small solution by $e^{z_1t/\eps}$ in \eqref{e.mainresult} so that something ``nontrivial'' can still be observed. In the parabolic setting, the renormalization in \eqref{e.conshe} can be naturally viewed as a shift of the height function by its average growing speed, after a Hopf-Cole transformation
\[
\log [u_\eps(t,x) e^{-\frac{c_1t}{\eps}-c_2t}]=\log u_\eps (t,x)-\frac{c_1t}{\eps}-c_2t.
\]
For the Schr\"odinger equation, it is less clear to us what should be the right physical quantity to look at. Another choice is to consider $\what{\phi}(t,\xi)$ for $\xi\sim O(1)$ and $t\sim O(\eps^{-\alpha})$ with some $\alpha>1$. In light of \eqref{e.kinetic} and the long time behavior of \eqref{e.wigner} analyzed in \cite{KR}, we expect some diffusion equation to show up in the limit.
\end{remark}
 
 \begin{remark}
 The convergences in \eqref{e.mainresult}, \eqref{e.homo} and \eqref{e.kinetic} hold in all dimensions $d\geq1$. In other words, if we start from the microscopic dynamics \eqref{e.maineq1}, with a random potential of size $\sqrt{\eps}$, then in all dimensions: (i) on the time scale $t/\eps$, depending on the initial data, we have either \eqref{e.homo} or \eqref{e.kinetic}; (ii) on the time scale $t/\eps^2$, if we have a low frequency initial data, then \eqref{e.mainresult} holds. The proof and the result does NOT depend on the dimensions. Nevertheless, with a random potential of size $\sqrt{\eps}$, the change of variables $(t,x)\mapsto (\tfrac{t}{\eps^2},\tfrac{x}{\eps})$ chosen in \eqref{e.maineq2} only leads to a spacetime white noise in $d=1$.
 \end{remark}
 
 \begin{remark}
 When the spacetime potential $V(t,x)$ is replaced by a spatial potential $V(x)$, similar problems (including nonlinear ones) have been analyzed in \cite{allez2015continuous,debussche2017solution,debussche2018schrodinger,Gu2017aa,gubinelli,cyril,zhang2012convergence} in $d=1,2,3$.
 \end{remark}

\section{Proofs}
\label{s.proof}

The proof contains two steps. First, we derive a probabilistic representation of the average wave function $\E[\what{\phi}(t,\xi)]$ with some auxiliary 
Brownian motion $\{B_t\}_{t\geq0}$ built on another probability space $(\Sigma,\mathcal{A},\Pb_\B)$. Using this probabilistic representation, we pass to the limit using tools from stochastic analysis. Similar proofs have already appeared in \cite{Gu2017aa,gu2018another}.

\subsection{Probabilistic representation}
Assuming $\{B_t\}_{t\geq 0}$ is a standard Brownian motion starting from the origin, defined on $(\Sigma,\mathcal{A},\Pb_\B)$. We denote the expectation with respect to $\{B_t\}_{t\geq0}$ by $\E_\B$.
\begin{lemma}\label{l.fk}
For the equation 
\begin{equation}\label{e.fkeq1}
i\partial_t \psi+\frac12\Delta\psi-V(t,x)\psi=0, \   \ t>0, x\in\R,
\end{equation}
with $\psi(0,x)=\psi_0(x)$, we have 
\begin{equation}\label{e.fkre}
\E[\what{\psi}(t,\xi)]=\what{\psi}_0(\xi)\E_\B[ e^{i\sqrt{i}\xi B_t} e^{-\frac12\int_0^t\int_0^t R(s-u,\sqrt{i}(B_s-B_u))dsdu}].
\end{equation}
\end{lemma}

On the formal level, \eqref{e.fkre} comes from an application of the Feynman-Kac formula to \eqref{e.fkeq1} then averaging with respect to $V$. We write \eqref{e.fkeq1} as 
\[
\partial_t \psi=\frac{i}{2}\Delta \psi-iV(t,x)\psi=0,
\]
and assume the following expression:
\[
\psi(t,x)=\E_\B[\psi_0(x+\sqrt{i}B_t)e^{-i\int_0^t V(t-s,x+\sqrt{i}B_s)ds}].
\]
Averaging with respect to $V$ and using the Gaussianity yields 
\[
\E[\psi(t,x)]=\E_\B[\psi_0(x+\sqrt{i}B_t) e^{-\frac12\int_0^t\int_0^t R(s-u,\sqrt{i}(B_s-B_u))dsdu}],
\]
which, after taking the Fourier transform, gives \eqref{e.fkre}.

\begin{proof}
We follow the proof of \cite[Proposition 2.1]{Gu2017aa}, where a similar formula is derived for spatial random potentials. For the convenience of readers, we provide all the details here. 

Fix $(t,\xi)$, we define the function
\[
F_1(z):=\E_\B[ e^{iz\xi B_t-\frac12\int_0^t\int_0^t R(s-u,z(B_s-B_u))dsdu}],\     \  z\in \bar{D}_0,
\]
with $D_0:=\{z\in \C: \mathrm{Re}[z^2]>0\}$. We also define the corresponding Taylor expansion
\[
F_2(z)=\sum_{n=0}^\infty F_{2,n}(z),\quad z\in \bar{D}_0,
\]
with
\[
F_{2,n}(z) :=\frac{(-1)^n}{2^n(2\pi)^{n}n!} \int_{[0,t]^{2n}}\int_{\R^{n}}\prod_{j=1}^n \what{R}(s_j-u_j,p_j) \E_\B\left[ e^{iz\xi B_t}  \prod_{j=1}^ne^{izp_j (B_{s_j}-B_{u_j})}\right] dpdsdu.
\]
Recall that $R(t,x)=\tfrac{R_\eta(t)}{\sqrt{2\pi}} e^{-x^2/2}$. In the definition of $F_1$, we have extended the definition so that  $R(t,z)=\tfrac{R_\eta(t)}{\sqrt{2\pi}} e^{-z^2/2}$ for all $z\in\C$. We also emphasize that $\what{R}(t,p)$ is the Fourier transform of $R(t,x)$ in the $x-$variable: 
\[
\what{R}(t,p)=R_\eta(t)e^{-\frac12p^2}.
\]

 It is straightforward to check that both 
$F_1$ and $F_2$ are analytic on $ D_0$ and continuous on $\bar{
  D}_0$. Note that $\sqrt{i}\in\partial  D_0$.
The goal is to show that 
\begin{equation}\label{may1806}
\E[\what{\psi}(t,\xi)]=\what{\psi}_0(\xi)F_1(\sqrt{i}).
\end{equation}
Since $(z,s,u)\mapsto R(s-u,z(B_s-B_u))$ is  bounded on $ \bar{ D}_0\times\R_+^2$, we have 
\begin{equation}
\begin{aligned}
F_1(z)=&\sum_{n=0}^\infty \frac{(-1)^n}{2^n n!} \E_\B\left[ e^{iz\xi B_t} \left(\int_{[0,t]^{2}}R(s-u,z(B_s-B_u))dsdu\right)^n\right]\\
=&\sum_{n=0}^\infty \frac{(-1)^n}{2^n n!} \E_\B\left[ e^{iz\xi B_t} \int_{[0,t]^{2n}}\prod_{j=1}^n R(s_j-u_j,z(B_{s_j}-B_{u_j}))dsdu\right]\\
=&\sum_{n=0}^\infty \frac{(-1)^n}{2^n(2\pi)^{n}n!} \E_\B\left[
  e^{iz\xi B_t}
  \int_{[0,t]^{2n}}\int_{\R^{n}}\prod_{j=1}^n \what{R}(s_j-u_j,p_j)
  e^{izp_j(B_{s_j}-B_{u_j})}dp dsdu\right].
\end{aligned}
\end{equation}
For $z=x\in\R$, we can apply the Fubini theorem to see that
$
F_1(x)=F_2(x).
$
Due to the analyticity and continuity of
$F_1$ and $F_2$, we therefore have $F_1(z)=F_2(z)$ for all $z\in \bar{ D}_0$. 
Hence,~(\ref{may1806}) is equivalent to  
\begin{equation}
\label{022204}
\E[\what{\psi}(t,\xi)]=\what{\psi}_0(\xi)\sum_{n=0}^\infty
F_{2,n}(\sqrt{i}).
\end{equation}
For a fixed $n$, we rewrite 
\[
\begin{aligned}
F_{2,n}(\sqrt{i})=\frac{(-1)^n}{2^n(2\pi)^{n}n!} \int_{[0,t]^{2n}}\int_{\R^{2n}}&
\prod_{j=1}^n \what{R}(s_{2j-1}-s_{2j},p_{2j-1})\delta(p_{2j-1}+p_{2j})\\
&\times  \E_\B\left[ e^{i\sqrt{i}\xi  B_t} e^{-\sum_{j=1}^{2n} i\sqrt{i}p_jB_{s_j}}\right]dsdp.
\end{aligned}
\]
Let $\sigma$ denote the permutations of $\{1,\ldots,2n\}$. After a relabeling of the $p$-variables we can write 
\begin{equation}\label{may1802}
\begin{aligned}
F_{2,n}(\sqrt{i}) =\frac{(-1)^n}{2^n(2\pi)^{n}n!} \sum_{\sigma}\int_{[0,t]_<^{2n}}\int_{\R^{2n}}&\prod_{j=1}^n \what{R}(s_{\sigma(2j-1)}-s_{\sigma(2j)},p_{\sigma(2j-1)})\delta(p_{\sigma(2j-1)}+p_{\sigma(2j)})\\
&\times  \E_\B\left[ e^{i\sqrt{i}\xi B_t} e^{-\sum_{j=1}^{2n} i\sqrt{i}p_jB_{s_j}}\right]ds dp,
\end{aligned}
\end{equation}
where $[0,t]_<^{2n}:=\{(s_1,\ldots,s_{2n}): 0\leq s_{2n}\leq\ldots\leq s_{1}\leq t\}$. Let $\F$ denote the pairings formed over $\{1,\ldots,2n\}$. 
It is straightforward to check that 
\begin{equation}\label{may1804}
\begin{aligned}
F_{2,n}(\sqrt{i}) =\frac{1}{i^{2n}(2\pi)^{n}} \sum_{\F}\int_{[0,t]_<^{2n}}\int_{\R^{2n}}&\prod_{(k,l)\in \F} \what{R}(s_k-s_l,p_k)\delta(p_k+p_l)\\
&\times  \E_\B\left[ e^{i\sqrt{i}\xi B_t} e^{-\sum_{j=1}^{2n} i\sqrt{i}p_jB_{s_j}}\right]ds dp.
\end{aligned}
\end{equation}
The pre-factors in (\ref{may1802}) and (\ref{may1804}) differ by a factor of $2^nn!$ since $i^{-2n}=(-1)^n$, and this comes from the mapping between the sets of permutations and pairings: for a given pairing with $n$ pairs, we have $n!$ ways of permutating the pairs, and inside each pair, we have $2$ options which leads to the additional factor of $2^n$. 

The phase factor inside the  integral in (\ref{may1804}) can be computed explicitly:
\begin{equation}\label{may1812}
\E_\B\left[ e^{i\sqrt{i}\xi B_t} e^{-\sum_{j=1}^{2n} i\sqrt{i}p_jB_{s_j}}\right]=e^{-\frac{i}{2}|\xi|^2(t-s_1)-\frac{i}{2}|\xi-p_1|^2(s_1-s_2)-\ldots-\frac{i}{2}|\xi-\ldots-p_{2n}|^2s_{2n}}.
\end{equation}

On the other hand, the equation \eqref{e.fkeq1} is written in the Fourier domain as 
\[
\partial_t \what{\psi}=-\frac{i}{2}|\xi|^2 \what{\psi}+\int_\R\frac{\what{V}(t,dp)}{2\pi i}\what{\psi}(t,\xi-p),   \    \ \what{\psi}(0,\xi)=\what{\psi}_0(\xi),
\]
where $V(t,x)$ admits the spectral representation $V(t,x)=\int_\R \tfrac{\what{V}(t,dp)}{2\pi}e^{ipx}$. Using the above formula, we can write the solution $\what{\psi}(t,\xi)$ as an infinite series 
\begin{equation}\label{may1810}
\begin{aligned}
\what{\psi}(t,\xi)=\sum_{n=0}^\infty \int_{[0,t]_<^n}\int_{\R^{n}}& \prod_{j=1}^n \frac{\what{V}(s_j,dp_j)}{2\pi i} e^{-\frac{i}{2}|\xi|^2(t-s_1)-\frac{i}{2}|\xi-p_1|^2(s_1-s_2)-\ldots-\frac{i}{2}|\xi-\ldots-p_{n}|^2s_{n}}\\
&\times \what{\psi}_0(\xi-p_1-\ldots-p_n) ds.
\end{aligned}
\end{equation}
 Evaluating the expectation $\E [\what{\psi}(t,\xi)]$ in \eqref{may1810}, using the Wick
formula for computing the Gaussian moment 
\[
\E[\what{V}(s_1,dp_1)\ldots\what{V}(s_n,dp_n)],
\]
and the  fact that 
\[
\E[\what{V}(s_i,dp_i)\what{V}(s_j,dp_j)]=2\pi \what{R}(s_i-s_j,p_i)\delta(p_i+p_j)dp_idp_j,
\]
and comparing the result to (\ref{may1804})-(\ref{may1812}), 
we conclude that \eqref{022204} holds, which completes the proof.
\end{proof}

\subsection{Convergence of Brownian functionals}
By Lemma~\ref{l.fk}, the interested quantity is written as
\[
\E[\what{\phi}_\eps(t,\xi)]=\what{\phi}_0(\xi)\E_\B[e^{i\sqrt{i}\xi B_t} e^{-\frac12\int_0^t\int_0^t R_\eps(s-u,\sqrt{i}(B_s-B_u))dsdu}],
\]
with $R_\eps$ defined as the covariance function of $\eps^{-3/2}V(t/\eps^2,x/\eps)$:
\[
R_\eps(t,x)=\frac{1}{\eps^3}R(\frac{t}{\eps^2},\frac{x}{\eps}).
\]
After a change of variable and using the scaling property of the Brownian motion, we have 
\[
\begin{aligned}
&\int_0^t\int_0^t R_\eps(s-u,\sqrt{i}(B_s-B_u))dsdu\\
&=\eps\int_0^{t/\eps^2}\int_0^{t/\eps^2}R_\eta(s-u)q(\sqrt{i}( B_{\eps^2s}- B_{\eps^2u})/\eps)dsdu\\
& \stackrel{\text{law}}{=}\eps\int_0^{t/\eps^2}\int_0^{t/\eps^2}R_\eta(s-u)q(\sqrt{i}(B_s-B_u))dsdu,
\end{aligned}
\]
where $R_\eta$ and $q$ were defined in \eqref{e.defq}. Thus, by defining  
\begin{equation}\label{e.defXepst}
X^\eps_t:=\frac{\eps}{2}\int_0^{t/\eps^2}\int_0^{t/\eps^2}R_\eta(s-u)q(\sqrt{i}(B_s-B_u))dsdu,
\end{equation}
we have
\begin{equation}\label{e.prore}
\E[\what{\phi}_\eps(t,\xi)]=\what{\phi}_0(\xi) \E_\B[ e^{i\sqrt{i}\xi \eps B_{t/\eps^2}-X^\eps_t}].
\end{equation}

To pass to the limit of $\E[\what{\phi}_\eps(t,\xi)]$, it suffices to prove the weak convergence of the random vector $(\eps B_{t/\eps^2}, X_t^\eps)$ (for fixed $t>0$) and a uniform integrability condition. The proof of Theorem~\ref{t.mainth} reduces to the following three lemmas.

\begin{lemma}\label{l.mean}
$\E_\B[X^\eps_t]=\frac{z_1t}{\eps}+O(\eps)$ with $z_1$ defined in \eqref{e.defz1}.
\end{lemma}

\begin{lemma}\label{l.wkcon}
For fixed $t>0$, as $\eps\to0$, 
\begin{equation}\label{e.inv}
(\eps B_{t/\eps^2}, X_t^\eps-\E_\B[X_t^\eps])\Rightarrow (N_1, N_2+iN_3)
\end{equation}
in distribution, where $N_1\sim N(0,t)$ and is independent of $(N_2,N_3)\sim N(0,t\mathsf{A})$, with the $2\times 2$ covariance matrix $\mathsf{A}$ defined in \eqref{e.coma}.
\end{lemma}

\begin{lemma}\label{l.uniforminte}
For any $\lambda\in\R$, there exists a constant $C>0$ such that 
\[
\E_\B[|e^{\lambda (X_t^\eps-\E_\B[X_t^\eps])}|]\leq C
\]
uniformly in $\eps>0$.
\end{lemma}

\begin{remark}
With some extra work as in \cite[Proposition 2.3]{gu2018another}, the convergence in \eqref{e.inv} can be upgraded to the process level. To keep the argument short, we only consider the marginal distributions, which is enough for the proof of Theorem~\ref{t.mainth}.
\end{remark}

\begin{proof}[Proof of Lemma~\ref{l.mean}]
A straightforward calculation gives
\[
\begin{aligned}
\E_\B[X_t^\eps]=&\eps\int_0^{t/\eps^2}ds \int_0^s   \frac{R_\eta(s-u)}{\sqrt{2\pi}} \E_\B[e^{-\frac{i}{2}|B_s-B_u|^2}] du\\
=&\eps\int_0^{t/\eps^2}ds \int_0^s \frac{R_\eta(u) }{\sqrt{2\pi}} \E_\B[e^{-\frac{i}{2}|B_s-B_{s-u}|^2}] du.
\end{aligned}
\]
Since $R_\eta$ is compactly supported, it is clear that 
\[
\E_\B[X_t^\eps]=\frac{z_1t}{\eps}+O(\eps),
\]
where 
\begin{equation}\label{e.defz1}
z_1=\int_0^\infty \frac{R_\eta(u)}{\sqrt{2\pi}} \E_\B[e^{-\frac{i}{2}|B_u|^2}] du=\int_0^\infty \frac{R_\eta(u)}{\sqrt{2\pi(1+iu)}}du.
\end{equation}
The proof is complete.
\end{proof}

\begin{proof}[Proof of Lemma~\ref{l.wkcon}]
The proof is based on a martingale decomposition. Denote the Brownian filtration by $\mathscr{F}_r$ and the Malliavin derivative with respect to $dB_r$ by $D_r$. An application of the Clark-Ocone formula \cite[Proposition 1.3.14]{nualart} leads to 
\[
\begin{aligned}
X_t^\eps-\E_\B[X_t^\eps]=\int_0^{t/\eps^2}\E_\B[D_rX_t^\eps |\mathscr{F}_r] dB_r.
\end{aligned}
\]
Recall that $X_t^\eps$ is defined in \eqref{e.defXepst}, by chain rule and the fact that 
\begin{equation}\label{e.3132}
D_r (B_s-B_u)=D_r \int_u^s dB_r =1_{[u,s]}(r),
\end{equation}
 we have
\[
D_rX_t^\eps=-i\eps\int_0^{t/\eps^2}\int_0^s \frac{R_\eta(s-u)}{\sqrt{2\pi}} e^{-\frac{i}{2}|B_s-B_u|^2}(B_s-B_u)1_{[u,s]}(r)duds, \quad r\in[0,t/\eps^2].
\]
Taking the conditional expectation with respect to $\mathscr{F}_r$ and computing the expectation 
\[
\E[ e^{-\frac{i}{2}X^2} X\,|\, B_r-B_u]
\] 
with $X\sim N(B_r-B_u,s-r)$ explicitly yields
\[
\begin{aligned}
Y_r^{\eps,t}&:=\eps^{-1}\E_\B[D_rX_t^\eps|\mathscr{F}_r]\\
&=-i \int_0^{t/\eps^2}\int_0^s \frac{R_\eta(s-u)}{\sqrt{2\pi}(1+i(s-r))^{3/2}}e^{-\frac{i|B_r-B_u|^2}{2(1+i(s-r))}}(B_r-B_u)1_{[u,s]}(r)duds.
\end{aligned}
\]
By the assumption, there exists $M>0$ such that $R_\eta(s-u)=0$ if $s-u\geq M$. Using the indicator function $1_{[u,s]}(r)$ in the above expression, we have for $M\leq r\leq t/\eps^2-M$ that 
\[
\begin{aligned}
Y_r^{\eps,t}=Y_r:
&=-i\int_r^{r+M}\int_{r-M}^r \frac{R_\eta(s-u)}{\sqrt{2\pi}(1+i(s-r))^{3/2}}e^{-\frac{i|B_r-B_u|^2}{2(1+i(s-r))}}(B_r-B_u)1_{[u,s]}(r)duds\\
&=-i \int_0^M\int_0^M\frac{R_\eta(s+u)}{\sqrt{2\pi}(1+is)^{3/2}}e^{-\frac{i|B_r-B_{r-u}|^2}{2(1+is)}}(B_r-B_{r-u})duds.
\end{aligned}
\]
The $Y_r$ defined above is only for $r\in [M,t/\eps^2-M]$, but we can extend the definition to $r\in\R$ by interpreting $B$ as a two-sided Brownian motion. Thus, by the fact that the Brownian motion has stationary and independent increments, we know $\{Y_r\}_{r\in\R}$ is a stationary process with a finite range of dependence. 

It is easy to check that 
\[
X_t^\eps-\E_\B[X_t^\eps]-\eps\int_0^{t/\eps^2}Y_rdB_r=\eps \int_0^{t/\eps^2} (Y_r^{\eps,t}-Y_r) dB_r\to0
\]
in probability. Define $Y_{1,r}=\mathrm{Re}[Y_r]$ and $Y_{2,r}=\mathrm{Im}[Y_r]$, applying Ergodic theorem, we have 
\[
\eps^2\int_0^{t/\eps^2} Y_{j,r}Y_{l,r} dr\to t\,\E[Y_{j,r}Y_{l,r}], \   \ j,l=1,2,
\]
and 
\[
 \eps^2\int_0^{t/\eps^2} Y_rds \to t\,\E[Y_r]=0,
\]
almost surely. We apply the martingale central limit theorem \cite[pp. 339]{ethier2009markov} to derive 
\[
(\eps B_{t/\eps^2}, \eps\int_0^{t/\eps^2} Y_rdB_r)\Rightarrow (B_t,W^1_t+iW^2_t)
\]
in $\mathcal{C}[0,\infty)$, where $B_t$ is a standard Brownian motion, independent of the two-dimensional Brownian motion $(W^1_t,W^2_t)$ with the covariance matrix $\mathsf{A}=(\mathsf{A}_{jl})_{j,l=1,2}$ given by
\begin{equation}\label{e.coma}
\mathsf{A}_{jl}=\E[Y_{j,r}Y_{l,r}].
\end{equation}
 The proof is complete.
\end{proof}

\begin{proof}[Proof of Lemma~\ref{l.uniforminte}]
We write 
\[
X_t^\eps-\E_\B[X_t^\eps]=\eps\int_0^{t/\eps^2}  \mathcal{Z}_sds,
\]
where 
\[
\mathcal{Z}_s:=\int_0^s\frac{ R_\eta(u)}{\sqrt{2\pi}} \left(e^{-\frac{i}{2}|B_s-B_{s-u}|^2}-\E_\B[e^{-\frac{i}{2}|B_s-B_{s-u}|^2}]\right)du.
\]
Again, assuming that $R_\eta(u)=0$ for $|u|\geq M$. Let $N_\eps=[\tfrac{t}{M\eps^2}]$, we have 
\[
\begin{aligned}
X_t^\eps-\E_\B[X_t^\eps]=&\eps \sum_{k=2}^{N_\eps} \int_{(k-1)M}^{kM} \mathcal{Z}_sds+\eps\left(\int_0^M+ \int_{N_\eps M}^{t/\eps^2}\right)\mathcal{Z}_sds\\
=&\eps \sum_{k=2}^{N_\eps}Z_k+\eps\left(\int_0^M+ \int_{N_\eps M}^{t/\eps^2}\right)\mathcal{Z}_sds
\end{aligned}
\]
where we defined $Z_k:=\int_{(k-1)M}^{kM} \mathcal{Z}_sds$ for $2\leq k\leq N_\eps$.  Since $\mathcal{Z}_s$ is uniformly bounded, we have 
\[
\left|\eps \left(\int_0^M+\int_{N_\eps M}^{t/\eps^2}\right)\mathcal{Z}_sds\right|\les \eps.
\]
For the first part, we write 
\[
\eps\sum_{k=2}^{N_\eps}Z_k=\left(\sum_{k\in A_{\eps,1}}+\sum_{k\in A_{\eps,2}}\right)\eps Z_k,
\]
with $A_{\eps,1}=\{2\leq k\leq N_\eps: k \mbox{ even }\}$ and $A_{\eps,2}=\{2\leq k\leq N_\eps: k \mbox{ odd }\}$. By the independence of the increments of the Brownian motion, we know that $\{Z_k\}_{k\in  A_{\eps,j}}$ are i.i.d. for $j=1$ and $2$. Therefore, 
\[
\begin{aligned}
\E_\B[|e^{\lambda (X_t^\eps-\E_\B[X_t^\eps])}|] \les &\E_\B[e^{\lambda \eps \sum_{k=2}^{N_\eps} \mathrm{Re}[Z_k]}]\\
\les & \sqrt{\E_\B[ e^{2\lambda \eps \sum_{k\in A_{\eps,1}} \mathrm{Re}[Z_k]}] \E_\B[ e^{2\lambda \eps \sum_{k\in A_{\eps,2}} \mathrm{Re}[Z_k]}]}.
\end{aligned}
\]
By the fact that $\{Z_k\}$ are bounded random variables with zero mean, we have for $j=1,2$ that 
\begin{equation}\label{e.3131}
\begin{aligned}
\E_\B[e^{2\lambda \eps \sum_{k\in A_{\eps,j}} \mathrm{Re}[Z_k]}]=& \prod_{k\in A_{\eps,j}} \E_\B[ e^{2\lambda \eps \mathrm{Re}[Z_k]}]\\
\leq &\prod_{k=1}^{N_\eps}\left(1+2\lambda^2\eps^2 \E_\B[|\mathrm{Re}[Z_k]|^2]+O(\eps^3) \right) \les 1.
\end{aligned}
\end{equation}
The proof is complete.
\end{proof}

\subsection{Proof of Theorem~\ref{t.mainth}.} By \eqref{e.prore}, we have 
\begin{equation}\label{e.prore2}
\E[\what{\phi}_\eps(t,\xi)]e^{\E_\B[X_t^\eps]}=\what{\phi}_0(\xi)\E_\B[ e^{i\sqrt{i}\xi \eps B_{t/\eps^2}-(X_t^\eps-\E_\B[X_t^\eps])}].
\end{equation}
By Lemma~\ref{l.wkcon}, we know that, for fixed $t>0,\xi\in\R$, the random variable 
\[
i\sqrt{i}\xi \eps B_{t/\eps^2}-(X_t^\eps-\E_\B[X_t^\eps])\Rightarrow i\sqrt{i}\xi N_1-(N_2+iN_3)
\]
in distribution, where $N_1\sim N(0,t)$ independent of $(N_2,N_3)\sim N(0,t\mathsf{A})$. Since Lemma~\ref{l.uniforminte} provides the uniform integrability:
\[
\E_\B[ |e^{i\sqrt{i}\xi \eps B_{t/\eps^2}-(X_t^\eps-\E_\B[X_t^\eps])}|^2]\leq \sqrt{\E_\B[|e^{i\sqrt{i}\xi \eps B_{t/\eps^2}}|^2]\E_\B[|e^{-(X_t^\eps-\E_\B[X_t^\eps])}|^2]}\les 1,
\] 
sending $\eps\to0$ on both sides of \eqref{e.prore2} and applying Lemma~\ref{l.mean}, we have 
\[
\begin{aligned}
\E[\what{\phi}_\eps(t,\xi)]e^{\frac{z_1t}{\eps}}\to \what{\phi}_0(\xi)\E_\B[ e^{i\sqrt{i}\xi N_1-(N_2+iN_3)}]=\what{\phi}_0(\xi)e^{-\frac{i}{2}|\xi|^2t} e^{\frac12(\mathsf{A}_{11}-\mathsf{A}_{22}+2i\mathsf{A}_{12})t}.
\end{aligned}
\]
Define 
\begin{equation}\label{e.defz2}
z_2=\frac12(\mathsf{A}_{11}-\mathsf{A}_{22}+2i\mathsf{A}_{12}),
\end{equation}
the proof of Theorem~\ref{t.mainth} is complete.

\subsection*{Acknowledgments} We would like to thank Lenya Ryzhik for asking this question and stimulating discussions. We also thank the two anonymous referees for a careful reading of the manuscript and for pointing out several possible
improvements as well as a technical mistake in the original manuscript. The work is partially supported by the NSF grant DMS-1613301/1807748 and the Center for Nonlinear Analysis at CMU.


\newcommand{\noop}[1]{} \def\cprime{$'$}

\end{document}